\documentclass[11pt]{article} 

\usepackage[utf8]{inputenc} 

\usepackage{geometry} 
\usepackage{soul}
\usepackage{graphicx} 
\usepackage{amsmath}
\usepackage{amssymb}
\usepackage{mathrsfs}

\usepackage{booktabs} 
\usepackage{array} 
\usepackage{paralist} 
\usepackage{verbatim} 
\usepackage{subfig} 

\usepackage{fancyhdr} 
\usepackage{sectsty}
\allsectionsfont{\sffamily\mdseries\upshape} 

\usepackage[nottoc,notlof,notlot]{tocbibind} 
\usepackage[titles,subfigure]{tocloft} 

\usepackage{amsthm}
\usepackage{color}
\newtheorem{theorem}{Theorem}

\newtheorem{corollary}{Corollary}

\newtheorem{lemma}{Lemma}

\newcounter{case}[theorem]
\newcounter{subcase}[case]
\newcounter{claim}[case]
\newcounter{subclaim}[claim]

\newenvironment{case}[1][]{\refstepcounter{case}\par\noindent\textbf{Case \thecase:}\space#1}{}

\begin{document}

\title{Representativity and waist of cable knots}
\author{Rom\'an Aranda, Seungwon Kim and Maggy Tomova}
\date{} 

\maketitle

\begin{abstract}
We study the incompressible surfaces in the exterior of a cable knot and use this to compute the representativity and waist of most cable knots.
\end{abstract}

\section{Introduction and definitions}

Let $K\subset S^{3}$ be a knot and let $S$ be a closed orientable
surface containing $K$. Following \cite{Ozawa}, we define the representativity
of the pair $(S,K)$ as the minimal intersection number $|K\cap\partial D|$
over all the compressing discs for $S$ in $S^{3}$. Denoted as $r(S,K)$,
the representativity of $(S,K)$ measures how many times the knot
is ``wrapping" the surface $S$. 

The \emph{representativity of $K$} is the maximal number of $r(F,K)$
among all the closed orientable surfaces $F\subset S^{3}$ that contain
the knot. In other words, 
\[
r(K)=\max_{K\subset F}\min_{D}|\partial D\cap K|
\]

This knot invariant had been studied by M. Ozawa and it is known for several classes of knots: torus knots, 2-bridge knots, and composite knots. There are also bounds for algebraic knots and knots with Conway spheres, see \cite{Ozawa}. Recently, Kindred determined that all alternating knots have representativity 2, \cite{Kind}. 
%

Similarly, if $F\subset S^{3}$ is a closed incompressible
surface in the exterior of a knot $K$, the waist of $K$, $waist(F,K)$ of $(F,K)$
is the minimum number of intersection $|D\cap K|$, between $K$ and
the compressing discs for $F$ in $S^{3}$. The \emph{waist} of $K$ is the maximum number $waist(F,K)$ among all the closed incompresible
surfaces in the exterior of $K$. In other words, 
\[
waist(K)=\max_{K\subset F}\min_{D}|\partial D\cap K|
\]

It is known that there are many classes of knots with waist one:
2-bridge knots \cite{inc surfs 2bridge knots}, torus knots \cite{torus knots},
twisted torus knots with twists on 2-strands \cite{twisted torus knots},
small knots, alternating knots \cite{alternating knots}, almost alternating
knots \cite{almost alternating knots}, toroidally alternating knots
\cite{toroidally alternating knots}, 3-braid knots \cite{braid knots},
Montesinos knots \cite{Montesinos knots}, and algebraically alternating
knots \cite{Algebraically alternating knots}. 

Let $V$ a solid torus in $S^3$. A \emph{cable knot} $K$ is an embedded circle in $\partial V$ with slope $p/q$ with respect to the Seifert framing for $V$, such that $gcd(p,q)=1$ and $q>1$. The number $q$ is called the \emph{index} of $K$ and $V$ is called the \emph{companion} of $K$. 

In this paper we study the behavior of both invariants, representativity
and waist, under the cabling operation. Let $K$ be a $(p,q)$-cable with companion solid torus $V$. The boundary torus of $V$ allows us to obtain the bounds $r(K)\geq p$ and $waist(K)\geq p\cdot waist(J)$, where $J$ is a core of $V$. 
We will show that, most of the time (see below), these estimates are exact. 

A cable knot is called \emph{inconsistent} if the knot on the companion torus is not a boundary slope of the companion knot; i.e., a cable knot is inconsistent if there is no essential surface in the complement of the torus whose boundary is the knot.

\begin{theorem} \label{thm representativity}
Let $K$ be an inconsistent cable knot with index $p$. Then $r(K) = p$ and the companion torus is the unique surface that realizes the representativity. \end{theorem}

In \cite{Pardon}, Pardon showed that the distortion of a knot, $\delta(K)$, is at least $\frac{1}{160}r(K)$. Thus, the following corollary follows immediately:

\begin{corollary}
Let $K$ be an inconsistent cable knot with index $p$. Then $\delta(K) \geq \frac{p}{160}$. 
\end{corollary}

\begin{theorem} \label{thm waist}
Let $K$ be an inconsistent cable knot with index $p$. Then $waist(K)=p\cdot waist(J)$ where $J$ is a companion knot for $K$. 
\end{theorem}

Recall that by \cite{S(K) is finite}, the set of boundary slopes for any knot is finite therefore the above theorem applies to almost every cable.

In particular for $(p,q)$-torus knots the set of all boundary slopes is $\{0,pq\}$ and for 2-bridge knots the set of boundary slopes is a subset of the even integers, \cite{inc surfs 2bridge knots}.  

\begin{corollary}
Let $J$ be either a 2-bridge or torus knot. Then for every cable $K$ of index $p$
along $J$, $r(K)=p$ and $waist(K)=p\cdot waist(K)$. 
\end{corollary}

\begin{proof}
Every cable knot $K$ with a pattern that is a 2-bridge knot or a torus knot has a finite, non-integer slope therefore $K$ is inconsistent.
\end{proof}

\section{Main results}

A properly embedded surface $F$ in a 3-manifold $M$ is \emph{incompressible} if it does not have any compressing discs. The surface is \emph{peripheral} if it is boundary parallel. A sphere in a 3-manifold is \emph{essential} if it does not bound a ball. The following two lemmas are well known:

\begin{lemma}\label{solidtorus}\cite{W1}
Every connected orientable incompressible surface in a solid torus is either a peripheral disc, a peripheral annulus, or a meridian disc.
\end{lemma}

\begin{lemma}\label{thickenedtorus}\cite{W2}
Suppose $F$ is a connected, orientable, incompressible surface properly embedded in a thickened torus $T \times I$. Then $F$ is one of the following:

\begin{enumerate}
\item A peripheral disc,
\item A peripheral annulus,
\item $\gamma \times I$ where $\gamma$ is an essential simple loop of $T$,
\item $T \times \{i\}$, where $i \in I$.

\end{enumerate} 

\end{lemma}

We will often use the following set up. Suppose $K$ is a cable knot with companion torus $V$ and let $T=\partial V$. Let $\eta(K)$ be an open regular neighborhood of $K$ in $S^3$ and $E(K) = S^3 - \eta(K)$. Consider a regular neighborhood of $T$ in $S^3$. The boundary of this neighborhood consists of two tori, $\tilde T$ will be the component contained outside of $V$. Let $\tilde V$ be the 3-manifold bounded by $\tilde T$ that contains $V$ with $\eta (K)$ removed. Observe that $\tilde V$ is a solid torus intersected with the exterior of $K$. Let $\tilde F=F-\tilde V$. Let $A_K = T - \eta(K)$.
The boundary of $A_K$ partitions $\partial E(K)$ into two annuli $\partial^+ E(K)$ and $\partial^- E(K)$. Let $T^{\pm} = A_K \cup \partial^{\pm} E(K)$. Without loss of generality, we assume that $T^-$ is the torus which bounds a solid torus $V^- \subset V$ which does not contains $K$. Also let $W$ be a thickened torus which is bounded by $T^+$ and $\tilde T$. Let $F^W = F \cap W$ and $F^- = F \cap V^-$.

\begin{lemma}\label{lem:essential_piece}

Let $K$ be a cable knot with companion torus $V$ and let $F$ be an incompressible and boundary incompressible surface in $S^3-\eta(K)$ possibly with boundary. There exists an isotopy of $F$ which minimizes $(|F\cap \tilde T| ,|F \cap A_K|)$ so that $F$ satisfies the following:

\begin{enumerate}

\item {Every component of $\tilde F$ is incompressible and boundary incompressible in $S^3- \tilde V$.}

\item {Every component of $F^W$ is incompressible in $W$.}

\item{Every component of $F^-$ is incompressible in $V^-$.}

\item{Every component of $F \cap \tilde V$ is incompressible and boundary incompressible in $\tilde V$.}

\end{enumerate}
\end{lemma}

\begin{proof}


First, we show that we can isotope $F$ so that $\tilde F$ is incompressible in $S^3 - \tilde V$. 
Isotope $F$ so that $|F \cap \tilde T|$ is minimal. Suppose that $\tilde F$ is compressible in $S^3- \tilde V$ with a compressing disc $D$. Since $F$ is incompressible, $\partial D$ bounds a disc $D'$ in $F$ such that $D' \cap \tilde T \neq \emptyset$. Since $E(K)$ is irreducible, $D \cup D'$ bounds a ball. So we can isotope $D'$ to remove the intersection $D' \cap T$, reducing the intersection $\tilde T \cap F$, which contradicts $|F\cap \tilde T|$ is minimal. Every component of $\tilde F$ is not boundary compressible since the only properly embedded incompressible, boundary compressible surface in solid torus complement is peripheral annulus, which can be isotoped to reduce the number of intersections. 
 
Note that by the same argument, $F \cap \tilde V$ is incompressible in $\tilde V$.

Now, we show that $F \cap \tilde V$ is boundary incompressible. Suppose that $F \cap \tilde V$ is boundary compressible. Since $F$ is already boundary incompressible in $E(K)$, there is a boundary compressing disk $D$ with $\partial D = \alpha \cup \beta$, $\alpha \subset \tilde T$ and $\beta \subset F\cap \tilde V$ an essential arc. Since $|F\cap \tilde T|$ is minimal, $\alpha$ must connect the same component $\sigma$ of $F\cap \tilde T$, and because $\tilde T$ is a torus and $F$ is orientable, $\alpha \cup \sigma$ must bound a bigon on $\tilde T$. By slightly pushing into $\tilde V$ the union of this bigon with $D$, we obtain a disc for $F\cap \tilde V$. But $F \cap \tilde V$ is incompressible so exist a disc in $F\cap \tilde V$ with the same boundary that the latter compression, inducing a parallelism between $\beta$ into the boundary of $F \cap \tilde V$, a contradiction. 

Second, we show that we can isotope $F \cap \tilde V$ so that $F^W$ is incompressible in $W$. In order to do this, isotope $F\cap \tilde V$, fixing $F \cap \tilde T$ so that $|F\cap A_K|$ is minimal. Then, using the same argument of the beginning of the proof, $F^-$ is incompressible in $V^-$, using the fact that $|F \cap T^-|$ is minimal. It follows that $F^W$ is also incompressible in $W$.

\end{proof}

\begin{proof}[\textbf{Proof of Theorem \ref{thm representativity}}]

Suppose $K$ is an inconsistent cable knot with a companion solid torus $V$. We will use the notation we established in the paragraph before Lemma \ref{lem:essential_piece}. Let $F$ be a surface such that $K \subset F$ and assume $F$ has been isotoped to satisfy the conclusions of Lemma \ref{lem:essential_piece}.  As $K\subset F$ and $K \subset \tilde V$, there are two possibilities; either $F$ is contained in $\tilde V$ or it intersects $\tilde T$.

\smallskip

\begin{case}\label{case1}
$F \subset \tilde V$, i.e., $F \cap \tilde T = \emptyset$. In this case we will prove that $r(K, F) \leq p$.
\end{case}

\smallskip

Let $f : \tilde V \rightarrow S^1$ be a Morse function and assume that $K$ and $F$ are in Morse position and $K$ has no critical points with respect to $f$. The preimage of every regular value of $f$ is a meridian disc for $\tilde V$. Let $D_t$ be one such disc and consider its intersection with $F$. As $F \cap \tilde T= \emptyset$ this intersection consists entirely of simple closed curves. As $D_t$ contains exactly $p$ points of $K$ and $K \subset F$, there are exactly $p$ points of $F$ contained in the curves $F \cap D_t$. In particular, an innermost such curve contains at most $p$ points of $K$. Note that such an innermost curve bounds a compressing disc for $F$ and therefore $r(K, F) \leq p$. 
\smallskip

\begin{case}\label{case2}
$F \cap \tilde T \neq \emptyset$. In this case we will prove that $r(K, F) \leq 2$.

\end{case}
\smallskip

Consider $F \cap \tilde T$ which is a collection of simple closed curves. If any of these curves are inessential on $\tilde T$, an innermost such would bound a compressing disc for $F$ disjoint from $K$  and so we can assume that $F \cap \tilde T$ is a set of parallel loops essential in $\tilde T$. By Lemma \ref{lem:essential_piece} it follows $\tilde F$, $F^W$ and $F^-$ are all incompressible and $\tilde F$ is boundary incompressible.

Recall that $(|F\cap \tilde T| , |F \cap A_K|)$ has been minimized. 

\medskip
\textbf{Claim:}  $\partial A_K \cap F \neq \emptyset$. 

\emph{Proof of Claim:} Suppose $\partial A_K \cap F =\emptyset$. Note that $F$ intersects $A_K$ only in loops by the hypothesis of the Claim and these loops are essential by Lemma \ref{lem:essential_piece}.

The surface $F_W$ has boundary on both $\tilde T$ and $T_+$. By the connectivity of $F$ and Lemma \ref{thickenedtorus}, there exists a non-peripheral annulus. Since one of the boundaries of this annulus is isotopic to $K$, every boundary of $F_W$  on $T_+$ is isotopic to $K$. Therefore $F\cap \tilde T$ is isotopic to $K$. We showed that $F \cap \tilde T$ bounds a incompressible surface in $S^3 - \tilde V$, so each component of $F \cap \tilde T$ is a boundary slope of a companion knot. However, this is not possible because $K$ is inconsistent. 

\qed
\medskip

By the Claim, we may assume that $\partial A_K \cap F \neq \emptyset$.  Note that $A_K \cap F$ is a set of simple loops and properly embedded simple arcs on $A_K$. By the minimality assumption in Lemma \ref{lem:essential_piece}, all  loops of intersection must be essential. Suppose that there exists a simple arc in $A_K \cap F$ which bounds a disc in $A_K$. Then an outermost such arc bounds a boundary compressing disc for $F$, which implies that $r(K,F) \leq 1$. Hence, we can assume that every simple arc of intersection is essential, hence a spanning arc in $A_K$. This also implies that there are no essential simple loops of intersection.

 Let $\Gamma^{\pm} =F^{\pm}\cap T^{\pm}$. Then $\Gamma^{\pm}$ are sets of essential simple loops on the tori $T^{\pm}$. By Lemma \ref{lem:essential_piece} $F^-$ is incompressible in $V^-$ and therefore $\Gamma^-$ is a set of essential simple loops in $T^-$ and so $F^-$ is a set of either peripheral annuli, or meridian discs. 

Suppose that $F^-$ is a set of peripheral annuli. Take an outermost such annulus $\mathcal{A}$ so that it cuts $V^-$ into two solid tori, one of them not containing $F^-$. Notice that $\mathcal{B}=F\cap \eta(K)$  is also a peripheral annulus which coincides with $\mathcal{A}$ in pairs of consecutive arcs in $\partial ^- E(K)$. When pushing both $\mathcal{A}$ and $\mathcal{B}$ towards $\partial ^- E(K)$, either their projections on  $\partial ^- E(K)$ coincide or not. If the projections do not coincide, taking a curve like in Figure \ref{boundarycompressing2} we obtain a compression for $F$ intersecting $K$ once. If some of their projections coincide, since $\mathcal{A}$ and $\mathcal{B}$ agree on their boundaries, both parallelisms induce a compressing disc for $F$ which intersects $K$ once. Hence $r(K,F)=1$. 

\begin{figure}[h]
\centering

\includegraphics[scale=0.3]{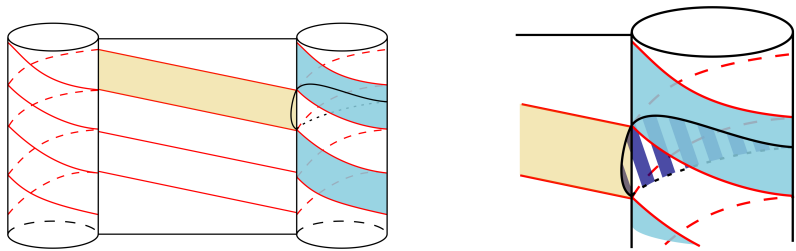}

\caption{The shaded rectangle in $A_K$ represents a piece of $\mathcal{A}$, and the curved rectangle in $\partial E(K)$ represents part of the projection of $\mathcal{B}$ on it. Observe that the black curve is contained in $\mathcal{A}\cup \mathcal{B}\subset F$; going inside $\eta (K)$ and interseting $K$ once. Here, $K$ can be thought of as the core of $\mathcal{B}$. Clearly it is a compression for $F$.}
\label{boundarycompressing2}
\end{figure}

Suppose now that $F^-$ is a set of meridian discs. Consider two adjacent meridian discs $D_1$ and $D_2$. The curves $\partial D_1$ and $\partial D_2$ cobound an annulus $\mathcal{A}\subset T^-$ with interior disjoint from $F$. As above, take the peripheral annulus $\mathcal{B}=F\cap \eta (K)$ and project it towards $\partial E(K)$. Let $C$ be a square component of $\mathcal{A}\cap A_K$; each of the arc components of $\partial C\cap \partial E(K)$ either coincide with the projection of $\mathcal{B}$ or not. If one arc coincides and the other does not, we can find a disc such that its boundary intersects $K$ geometrically and algebraically twice (see Figure \ref{gabaidisc1}). Hence, this disc is a compressing disc of $F$, and $r(K, F) \leq 2$. If none of the arcs coincide, we can find two discs which intersect $K$ geometrically twice but algebraically $0$ times (see figure \ref{gabaidisc2}). However, the boundaries of the two discs intersect once, so both are compressing discs of $F$. Hence, $r(K,F) \leq 2$. Notice that if both arcs of $\partial C\cap \partial E(K)$ coincide, we can replace $D_2$ with the other meridian disc in $F^-$ adjacent to $D_1$ and reach the same conclusion. Therefore $r(F,K)\leq 2$.\\

\begin{figure}[h]
\centering
\includegraphics[scale=0.35]{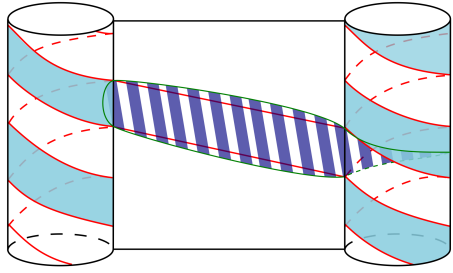}
\caption{The green loop is a boundary of a compressing disc of $F$ which intersects $K$ geometrically and algebraically twice.}
\label{gabaidisc1}
\end{figure}

\begin{figure}[h]
\centering
\includegraphics[scale=0.35]{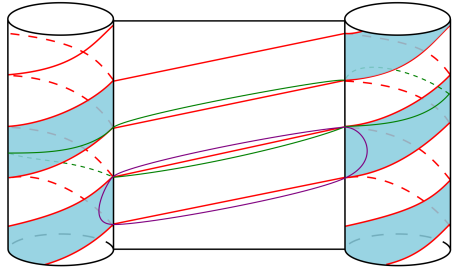}
\caption{The green and purple loops are boundaries of discs in the exterior of $F$ in $S^3$ which intersects $K$ geometrically twice but algebraically zero times. However, they are transversely intersect once, hence both loops are essential on $F$, hence, they are compressing discs of $F$.}
\label{gabaidisc2}
\end{figure}

\end{proof}

\begin{proof}[\textbf{Proof of Theorem \ref{thm waist}}]

Let $F\subset S^3$ be a closed surface disjoint from $K$ such that $F$ is incompressible in $S^3-\eta(K)$. We will continue to use the notation established in the paragraph before Lemma \ref{lem:essential_piece}. If $F\cap \tilde V=\emptyset$ then $F$ is contained in the complement of $\tilde V$. If $D$ is a compressing disc for $F$, it will intersect $\tilde V$ in meridians and so $|D\cap K|=p\cdot |D\cap J|$ where $J$ is the core of $V$. Thus $waist(K,F)=waist(J,F)\cdot p$.

Suppose now that $F\cap \tilde V\neq \emptyset$, and recall that $F^-$, $F^W$ and $\tilde F$ are incompressible in $V^-$, $W$ and $S^3-\tilde V$, respectively (Lemma \ref{lem:essential_piece}). Moreover, since $F$ is disjoint from $K$, $\partial(A_k)\cap F= \emptyset$. By Lemma \ref{thickenedtorus}, $F^W$ is either isotopic to $T\times \{i\}$ (in such case $waist(K,F)=p$), or $F^W$ contains a non-peripheral annulus. Suppose the latter, then the components of $F\cap T$ are isotopic to $K$ and, since $K$  is inconsistent, $\tilde F$ must be the union of boundary parallel annuli in $S^3-V$ which can be pushed inside $\tilde V$. Hence, $F^W \cup \tilde F\simeq F^W$ is the union of annuli parallel to annuli in $T^+$. Finally, Lemma \ref{solidtorus} implies that $F^-$ is union of peripheral annuli and so $F=F^W\cup F^-$ is parallel to $\partial \eta (K)$, and thus $waist(F,K)=1$.


\end{proof}

\end{document}